\documentclass[11pt,a4paper]{amsart}
\usepackage{mathpazo}  
\usepackage{amsmath}
\usepackage{amssymb}
\usepackage{stmaryrd}
\usepackage{graphicx}
\usepackage{setspace}
\usepackage{soul}
\usepackage{latexsym}
\usepackage{mathrsfs}
\usepackage{fancyhdr}
\usepackage{enumerate}
\usepackage[english]{babel}
\usepackage{lipsum}
\usepackage{etoolbox}
\usepackage{colonequals}
\usepackage[colorlinks=true,linkcolor=Maroon, citecolor=LimeGreen, pagebackref,linktocpage=true]{hyperref} 
\usepackage[usenames,dvipsnames]{xcolor}
\usepackage{lipsum}
\usepackage{mwe}
\usepackage{float}
\usepackage{mathrsfs}
\usepackage{cases}
\usepackage{amsfonts} 
\usepackage{geometry}
\usepackage{graphicx}
\usepackage{tikz}
\usepackage{amsmath}
\usepackage{amssymb}
\usepackage[all]{xy}

 \usepackage{tikz-cd}

\usepackage{color}
\usepackage{colordvi}
\usepackage{multicol}
\usepackage[linktocpage=true]{hyperref}
\begin{document}
\input xy
\xyoption{all}
\onehalfspacing
\begin{sloppypar}

\numberwithin{equation}{section}
\newcommand{\bA}{\mathbb{A}}
\newcommand{\bB}{\mathbb{B}}
\newcommand{\bC}{\mathbb{C}}
\newcommand{\bD}{\mathbb{D}}
\newcommand{\bE}{\mathbb{E}}
\newcommand{\bF}{\mathbb{F}}
\newcommand{\bG}{\mathbb{G}}
\newcommand{\bK}{\mathbb{K}}
\newcommand{\bL}{\mathbb{L}}
\newcommand{\bN}{\mathbb{N}}
\newcommand{\bQ}{\mathbb{Q}}
\newcommand{\bP}{\mathbb{P}}
\newcommand{\bX}{\mathbb{X}}
\newcommand{\bZ}{\mathbb{Z}}
\newcommand{\cA}{\mathcal{A}}
\newcommand{\cB}{\mathcal{B}}
\newcommand{\cC}{\mathcal{C}}
\newcommand{\mD}{\mathcal{D}}
\newcommand{\cF}{\mathcal{F}}
\newcommand{\mH}{\mathcal{H}}
\newcommand{\cI}{\mathcal{I}}
\newcommand{\cK}{\mathcal{K}}
\newcommand{\mL}{\mathcal{L}}
\newcommand{\cM}{\mathcal{M}}
\newcommand{\cN}{\mathcal{N}}
\newcommand{\cO}{\mathcal{O}}
\newcommand{\cP}{\mathcal{P}}
\newcommand{\cS}{\mathcal{S}}
\newcommand{\cT}{\mathcal{T}}
\newcommand{\cU}{\mathcal{U}}
\newcommand{\La}{\Lambda}
\newcommand{\ve}{\mathcal{E}}
\newcommand{\vf}{\varphi}
\newcommand{\tC}{\textbf{C}}
\newcommand{\QG}{\mathbf{U}}
\newcommand{\Hall}{\mathbf{H}}
\newcommand{\SHall}{\mathbf{SDH}}
\newcommand{\DDHall}{\mathbf{DDH}}
\newcommand{\DHall}{\mathbf{DH}}
\newcommand{\mul}{\mathbf{m}}
\newcommand{\id}{\mathbf{1}}
\newcommand{\tw}{\operatorname{tw}}
\newcommand{\cw}{\operatorname{cw}}
\newcommand{\ex}{\operatorname{ex}}
\newcommand{\red}{red}
\newcommand{\g}{{\mathfrak{g}}}
\newcommand{\wh}{\widehat}  
\newcommand{\HG}{\mathrm{H}}     
\newcommand{\Proj}{\ensuremath{\mathbb{P}}}
\newcommand{\Ext}{\mathrm{Ext}}
\newcommand{\End}{\mathrm{End}}
\newcommand{\Hom}{\mathrm{Hom}}
\newcommand{\add}{\mathrm{add}}
\newcommand{\ind}{\mathrm{ind}}
\newcommand{\rad}{\mathrm{rad}}
\newcommand{\modu}{\mathrm{mod}}
\newcommand{\Aut}{\mathrm{Aut}}
\newcommand{\Iso}{\mathrm{Iso}}
\newcommand{\Gr}{\mathrm{Gr}}
\newcommand{\soc}{\operatorname{{soc}}}
\newcommand{\supp}{\operatorname{supp}\nolimits}
\newcommand{\gl}{\operatorname{gl.dim}\nolimits}
\newcommand{\im}{\operatorname{Im}\nolimits}
\newcommand{\coker}{\operatorname{Coker}}
\newcommand{\pd}{\operatorname{proj.dim}\nolimits}
\newcommand{\injd}{\operatorname{inj.dim}\nolimits}
\newcommand{\ra}{{\rangle}}
\newcommand{\la}{{\langle}}
\newcommand{\D}{\ensuremath{\partial}}

\newtheorem{theorem}{Theorem}[section]
\newtheorem*{acknowledgement}{Acknowledgements}
\newtheorem{algorithm}[theorem]{Algorithm}
\newtheorem{axiom}[theorem]{Axiom}
\newtheorem{case}[theorem]{Case}
\newtheorem{claim}[theorem]{Claim}
\newtheorem{conclusion}[theorem]{Conclusion}
\newtheorem{condition}[theorem]{Condition}
\newtheorem{conjecture}[theorem]{Conjecture}
\newtheorem{construction}[theorem]{Construction}
\newtheorem{corollary}[theorem]{Corollary}
\newtheorem{criterion}[theorem]{Criterion}
\newtheorem{definition}[theorem]{Definition}
\newtheorem{example}[theorem]{Example}
\newtheorem{exercise}[theorem]{Exercise}
\newtheorem{lemma}[theorem]{Lemma}
\newtheorem{notation}[theorem]{Notation}
\newtheorem{problem}[theorem]{Problem}
\newtheorem{proposition}[theorem]{Proposition}
\newtheorem{solution}[theorem]{Solution}
\newtheorem{summary}[theorem]{Summary}
\newtheorem*{thm}{Theorem}
\newcommand{\qbinom}[2]{\begin{bmatrix} #1\\#2 \end{bmatrix} }
\newcommand{\nc}{\newcommand}

\allowdisplaybreaks
\theoremstyle{remark}
\newtheorem{remark}[theorem]{Remark}

\title[Semi-derived Ringel-Hall algebras and Drinfeld double]{Semi-derived Ringel-Hall bialgebras}

\author[Y. Li]{Yiyu Li}
\address{Department of Mathematics, Sichuan University, Chengdu 610064, P.R.China}
\email{liyiyumath@gmail.com}
\author[L. Peng]{Liangang Peng}
\address{Department of Mathematics, Sichuan University, Chengdu 610064, P.R.China}
\email{penglg@scu.edu.cn}

\subjclass[202  0]{18E30,16E60,17B37}

\begin{abstract}

Let $\cA$ be an arbitrary hereditary abelian category. Lu and Peng defined the semi-derived Ringel-Hall algebra $\SHall(\cA)$ of $\cA$ and proved that $\SHall(\cA)$ has a natural basis and is isomorphic to the Drinfeld double Ringel-Hall algebra of $\cA$. In this paper, we introduce a coproduct formula on $\SHall(\cA)$ with respect to the basis of $\SHall(\cA)$ and prove that this coproduct is compatible with the product of $\SHall(\cA)$, thereby the semi-derived Ringel-Hall algebra of $\cA$ is endowed with a bialgebra structure which is identified with the bialgebra structure of the Drinfeld double Ringel-Hall algebra of $\cA$.

\end{abstract}

\maketitle

\section{Introduction}\label{sec: intro}      

     Let $\bF_q$ be a field with $q$ elements and $\cA$ be an $\bF_q$-linear exact category. The Hall algebra $\Hall(\cA)$ of $\cA$ is an associative algebra whose underlying space is a $\bC$-vector space with a basis parametrized by the isomorphic classes of objects in $\cA$. The structure constants of the product are given by certain counts of extensions. To be more specific, for any $ M,N\in \cA$, the product $[M]\ast[N]$ is given by a linear combination $\sum_{[L]\in\cA}\cF_{MN}^L[L]$, where $\cF_{MN}^L$ is the counts of the possible extensions of $M$ by $N$ in $\cA$. The concept of Hall algebras first appeared in \cite{Stein} and was considered as a tool in dealing with some questions about abelian $p$-groups for a prime number $p$. 
     
     In the 1990s, Ringel developed the theory of Hall algebras of finitary abelian categories. Since then, Hall algebras have obtained a lot of attention and have become a useful tool to establish connections between representation theory  of finite dimensional algebras and other topics, such as Lie theory, combinatorics, and algebraic geometry, etc. For a complex simple Lie algebra $\g$, Ringel  \cite{CMR90} proved that the positive part of the quantum group $\QG_v(\g)$ \cite{Dr87} of $\g$ is isomorphic to the Hall algebra $\Hall(\cA)$ of a hereditary category $\cA$, where $\cA$ is the abelian category of $\bF_q$-representations of a quiver whose underlying graph is the same as the Dynkin diagram of $\g$. 
     
     Later  Green \cite{JG95} extended this result to the Kac-Moody type. In particular, Green provided a coproduct for the Hall algebra, and then endows the Hall algebra with a bialgerba structure based on the famous Green's formula. As a consequence,  the positive part of a quantum group of Kac-Moody type is isomorphic to the composition subalgebra of the Hall algebra of the category of $\bF_q$-representations of a  corresponding quiver.

     Inspired by the work of Ringel and Green, several attempts have been made to develop Ringel-Hall algebras from appropriate categories in order to realize the entire quantum groups. In \cite{JX97}, Xiao considered the Drinfeld double of the extended Ringel-Hall algebra of an acyclic quiver $Q$ and realized the whole quantum group associated to $Q$ as the composition subalgebra of the Drinfeld double algebra.
    
    For an abelian hereditary category $\cA$ with enough projective objects, Bridgeland \cite{Bri13} considered the Hall algebra of the exact category of $\bZ/2$-graded complexes with projective components over $\cA$, and then localized it at the isomorphism classes of acyclic complexes. He proved that a subalgebra of the resulting Hall algebra (called Bridgeland's Hall algebra) of an acyclic quiver is isomorphic to the corresponding quantum group. It is stated in \cite{Bri13} and proved in \cite{Ya16} that Bridgeland's Hall algebra is isomorphic to the Drinfeld double Hall algebra of $\cA$.
    
    Recently, Lu and Peng \cite{LP20} defined the semi-derived Ringel-Hall algebra of an arbitrary hereditary category $\cA$, denoted by $\SHall_{\bZ/2}(\cA)$, to be the localization of a quotient algebra of the Ringel-Hall algebra of the category of $\bZ/2$-graded complexes. They proved that $\SHall_{\bZ/2}(\cA)$ is isomorphic to Drinfeld double Hall algebra of $\cA$. The semi-derived Hall algebra also admits a basis, parameterized by the objects of $\cA$ and the elements of its Grothendieck group $K_0(\cA)$, which serves as the foundation for many computations throughout this paper. As a consequence,  a quantum generalized Kac-Moody algebra $\QG_v(\g)$ \cite{Kang95,Bor88} can be realized as the composition subalgebra of the semi-derived Ringel-Hall algebra of a corresponding  quiver associated to $\g$ \cite{Lu23}.
     
     The bialgebra structure of the extended Hall algebra induces a bialgebra structure for the Drinfeld double Hall algebra. According to the isomorphism between semi-derived Ringel-Hall algebra and Drinfeld double Hall algebra, it is natural to inquire whether an explicit formula exists for endowing $\SHall_{\bZ/2}(\cA)$ with a bialgebra structure. This paper is devoted to giving an affirmative answer to this question.
      
     The paper is organized as follows. After recalling the basics of Ringel-Hall algebras, Drinfeld double Hall algebras, and semi-derived Ringel-Hall algebras in Section \ref{sec: priminary}, we introduce an explicit coproduct formula for the semi-derived Ringel-Hall algebra in the basis of $\SHall_{\bZ/2}(\cA)$ for an arbitrary hereditary category $\cA$ and state the main result  in Section \ref{sec: bistructure}.
      As a consequence, the bialgebra structure of $\SHall_{\bZ/2}(\cA)$ is identified with the bialgebra structure of the Drinfeld double Hall algebra as well.
     The proof will be presented in Section \ref{sec: proof}.

     \begin{acknowledgement}
The authors would like to thank Prof. Ming Lu and Prof. Changjian Fu for their discussions and thoughtful comments, which led to an improvement of the paper while it was still in preparation. The authors are also grateful to the reviewers for their careful reading of the manuscript and valuable suggestions, which greatly helped to enhance the quality of this paper.
     \end{acknowledgement}

\section{Preliminaries}\label{sec: priminary}      
     Let $\bF_q$ be a finite field with $q\in\bZ$ elements and $\upsilon=q^{1/2}$.  An $\bF_q$-linear abelian category $\cA$ is called finitary if $|\Hom_{\cA}(M,N)|<\infty$ and $|\Ext^1_{\cA}(M,N)|<\infty$ for all $\  M,N\in\cA$. For an object $M\in\cA$, its projective dimension (denote by $\pd M$) is defined to be the smallest number $i\in\bN$ such that $\Ext_\cA^{i+1}(M,-)=0$; dually, one can define its injective dimension ($\injd M$). We say that $\cA$ is of global dimension at most $n$ if there exists a positive integer $n$ such that $\pd M\leq n$ and $\injd M\leq n$ for all $M\in\cA$. In this case, we denote it by $\gl\cA\leq n$. In particular, $\cA$ is called a hereditary category if $\gl\cA\leq 1$. Throughout this section, we always assume that $\cA$ is a small, finitary, abelian $\bF_q$-linear category of finite global dimension without further specification.
     
     Denote the Grothendieck group of $\cA$ by $K_0(\cA)$, for every object $M$ in $\cA$ we denote the representative of $M$ in $K_0(\cA)$ by $\wh{M}$. Since $\cA$ is a category of finite global dimension, a bilinear form $\la-,-\ra: K_0(\cA)\times K_0(\cA)\rightarrow\bZ$ on $K_0(\cA)$ is well-defined as follows:
     
     \[\la\wh{M},\wh{N}\ra:=\sum_{i\in\bZ}(-1)^i\dim\Ext_\cA^i(M,N),\]
      where $\Ext_\cA^0(M,N)=\Hom_\cA(M,N)$. We call $\la-,-\ra$ the Euler form of $\cA$. We also consider the symmetric Euler form:
     \[(-,-):K_0(\cA)\times K_0(\cA)\rightarrow\bZ,\]
      defined by $(\wh M,\wh N)=\la\wh M,\wh N\ra+\la\wh N,\wh M\ra$.

     \subsection{Ringel-Hall Algebras}\label{subsec: Ringel-Hall Alg}

     Denote by $\Iso(\cA)$ the set of all isomorphism classes $[M]$ of objects $M$ in $\cA$. For three objects $M,N,R\in\cA$, $\Ext_\cA^1(M,N)_R\subset\Ext_\cA^1(M,N)$ is the subset consisting of all extensions of M by N with the middle term isomorphic to $R$. We define the Hall algebra $\Hall(\cA)$ of $\cA$ to be a $\bC$-vector space:
     \[\Hall(\cA):=\bigoplus_{[M]\in\Iso(\cA)}\bC[M],\]
     with a product
     
    \[[M]\ast [N]=\sum\limits_{R\in\Iso(\cA)} h_{MN}^R[R].\]
     where $h_{MN}^R={|\Ext^1_{\cA}(M,N)_R|}/{|\Hom_{\cA}(M,N)|}$. In \cite{CMR90}, it has been proved that $\Hall(\cA)$ is an associative algebra with unit $[0]$. To be specific, for $M,N,L,K\in\cA$, we have 
     \begin{equation}\label{eq:AssOfHA}
       \sum_{R}h_{MN}^Rh_{RL}^K=\sum_{T}h_{MT}^Kh_{NL}^T,  
     \end{equation}
    which induces the associative law of $\Hall(\cA)$.

     The twisted Hall algebra $\Hall_{\tw}(\cA)$ is the $\upsilon$-deformation of the Hall algebra which has the same underlying space as $\Hall(\cA)$, and the product of $ \Hall_{\tw}(\cA)$ is given by
     \[[M]\ast[N]=\sum\limits_{R\in\Iso(\cA)} \upsilon^{\la\widehat{M},\widehat{N} \ra} h_{MN}^R[R].\]
     $\Hall_{\tw}(\cA)$ is clearly a unital associative algebra as well \cite{CMR90}. We denote by $\Hall_{\tw}(\cA)\widehat{\otimes}\Hall_{\tw}(\cA)$ the space of formal linear combinations
      \[\sum_{[A],[B]\in\Iso(\cA)}c_{A,B}[A]\otimes[B].\]
     We assume that $\cA$ is a small, finitary $\bF_q$-linear hereditary abelian category. In \cite{JG95}, Green introduced a coproduct and counit on $\Hall_{\tw}(\cA)$ such that $\Hall_{\tw}(\cA)$ is equipped with a (topological) bialgebra structure.

     Denote $a_M=|\Aut_{\cA}(M)|$ for any object $M$. For any $n\in \bZ_{\geq 0}, M_1, ..., M_n, K\in\cA$, we recursively define  $h_{M_1M_2\cdots M_n}^K$  as follows: for  $n = 3$ , we set  $h_{M_1M_2M_3}^K := \sum_{T} h_{M_1T}^K h_{M_2M_3}^T$ ; for  $n \geq 4$ , we define  $h_{M_1M_2\cdots M_n}^K := \sum_{T} h_{M_1T}^K h_{M_2\cdots M_n}^T$. 
     \begin{theorem}\cite{JG95}\label{thm: hall bialgebra}
         There exists a coproduct map $\Delta:\Hall_{\tw}(\cA)\rightarrow\Hall_{\tw}(\cA)\widehat{\otimes}\Hall_{\tw}(\cA)$ of the form
         \[\Delta([R])=\sum\limits_{M,N} \upsilon^{\la \widehat{M},\widehat{N} \ra} h_{MN}^R\frac{a_R}{a_Ma_N} [M]\otimes [N],\]
         and a counit map $\eta:\Hall_{\tw}(\cA)\rightarrow\bC$ that is given by
         \[\epsilon([M])=\delta_{[M],[0]},\]
         such that $(\Hall_{\tw}(\cA),\Delta,\epsilon )$ is a (topological) bialgebra.
     \end{theorem}

     The following lemma is crucial for the proof theorem \ref{thm: hall bialgebra}.
     \begin{lemma}
     \label{thm: green formula}[Green's Formula]
         Let $\cA$ be a hereditary category, for all $M,N,K_1,K_2 \in A$
         We have 
         $$\sum\limits_{K} a_Kh^{K}_{MN}h^{K}_{K_1K_2}=\sum\limits_{M_1,M_2,N_1,N_2} \upsilon^{-2\la \widehat{M_1},\widehat{N_2} \ra}\frac{a_Ma_Na_{K_1}a_{K_2}}{a_{M_1}a_{M_2}a_{N_1}a_{N_2}}h^M_{M_1M_2}h^N_{N_1N_2}h^{K_1}_{M_1N_1}h^{K_2}_{M_2N_2}.$$
     \end{lemma}
     The following is a direct consequence of Lemma \ref{thm: green formula} and equation (\ref{eq:AssOfHA}) which is useful in the proof of theorem \ref{thm: bialgebra}.

     \begin{corollary}\label{cor: green formula-3,2}
              Let $\cA$ be an hereditary category, for all $M,N,K_1,K_2,C \in A$, 
         we have 
         \begin{eqnarray*}
         &\sum\limits_{K}& a_Kh^{K}_{MN}h^K_{K_1CK_2}\\
&=&\sum\limits_{\substack{M_1,C_1,M_2\\N_1,N_2,C_2}}\upsilon^{-2\la \widehat{M_1},\widehat{C_2}+\widehat{N_2} \ra-2\la \widehat{C_1},\widehat{N_2} \ra}\frac{a_Ma_Na_{K_1}a_{K_2}a_C}{a_{M_1}a_{N_1}a_{C_1}a_{M_2}a_{C_2}a_{N_2}} h^M_{M_1C_1M_2}h^N_{N_1C_2N_2}h^{K_1}_{M_1N_1}h^C_{C_1C_2}h^{K_2}_{M_2N_2}.
         \end{eqnarray*}
     \end{corollary}

     \subsection{Extended Hall algebras and Drinfeld Double}For a hereditary category $\cA$, the extended Hall algebra $\Hall^{\ex}_{\tw}(\cA)^+$ of $\cA$ is an associative algebra with the underlying vector space $\Hall_{\tw}(\cA)\otimes_{\bC} \bC[K_0(\cA)]$, where $\bC[K_0(\cA)]$ is the group algebra of the Grothendieck group $K_0(\cA)$. Its product is defined to be
     \[[M]k_\alpha\ast[N]k_\beta=\sum\limits_{[R]\in\Iso(\cA)}\upsilon^{\la\widehat{M},\widehat{N}\ra+(\alpha,\wh{N})} h_{MN}^R[R]k_{\alpha+\beta},\]
     with the unit $[0]$. In \cite{JX97}, the extended Hall algebra $\Hall^{\ex}_{\tw}(\cA)^+$ has a bialgebra structure whose coproduct is given by
     \begin{align*}
     \Delta^+([R])=&\sum\limits_{M,N} \upsilon^{\la \widehat{M},\widehat{N} \ra} h_{MN}^R\frac{a_R}{a_Ma_N} [M]k_{\widehat{N}}\otimes [N],\\
     \Delta^+(k_\alpha)=& k_\alpha\otimes k_\alpha,
     \end{align*}
     while the counit is defined as  $\epsilon([R])=0$ and $\epsilon(k_\alpha)=1$ for $R\neq 0\in\cA$ and $ \alpha\in K_0(\cA)$. 
      Dually, let $\Hall^{\ex}_{\tw}(\cA)^-$ be the same vector space as the underlying space of $\Hall^{\ex}_{\tw}(\cA)^+$. The bialgebra structure of $\Hall^{\ex}_{\tw}(\cA)^-$ is given as follows
     \begin{align*}[M]k_\alpha\ast[N]k_\beta=&\sum\limits_{[R]\in\Iso(\cA)}\upsilon^{\la\widehat{M},\widehat{N}\ra-(\alpha,\wh{N})} h_{MN}^R[R]k_{\alpha+\beta},\\
     \Delta^-([R]):=&\sum\limits_{M,N} \upsilon^{\la \widehat{M},\widehat{N} \ra} h_{MN}^R\frac{a_R}{a_Ma_N} [N]\otimes [M]k_{\wh{N}},\\
     \Delta^-(k_\alpha)=& k_\alpha\otimes k_\alpha,
     \end{align*}
     with unit $[0]$ and counit $\epsilon([R])=0$ and $\epsilon(k_\alpha)=1$ for $R\neq 0\in\cA$ and $ \alpha\in K_0(\cA)$. We denote by $\Hall^{\ex}_{\tw}(\cA)^+\widehat{\otimes}\Hall^{\ex}_{\tw}(\cA)^-$ the space of formal linear combinations
      \[\sum_{[A],[B]\in\Iso(\cA),\alpha,\beta\in K_0{(\cA)}}c_{A,B}[A]k_\alpha\otimes[B]k_\beta.\]
      There is a bilinear pairing $\varphi:\Hall^{\ex}_{\tw}(\cA)^+\times\Hall^{\ex}_{\tw}(\cA)^-\rightarrow\bC$  given by
      \[\varphi([M]k_\alpha, [N]k_\beta)=\upsilon^{(\alpha,\beta)}\delta_{[M],[N]}a_M,\]
      which is a Hopf pairing \cite{Jo95} ,i.e., it satisfies
       \[\varphi([M]k_\alpha\ast[N]k_\beta,[R]k_\delta)=\vf([M]k_\alpha\otimes[N]k_\beta,\Delta([R]k_\delta)),\]
      where $\varphi(x\otimes x', y\otimes y')=\varphi(x,y)\varphi(x',y')$ for $x,x'\in\Hall^{\ex}_{\tw}(\cA)^+,y,y'\in \Hall^{\ex}_{\tw}(\cA)^-$. As a consequence, we have the following result.
      \begin{theorem}\cite{JX97,Jo95}\label{thm: DDH(A)}
      There is an algebra structure on $\Hall^{\ex}_{\tw}(\cA)\widehat{\otimes}\Hall^{\ex}_{\tw}(\cA)$ given by
      \begin{itemize}
      \item[(D1)] $(x\otimes 1)(x'\otimes 1)=xx'\otimes 1,$
      \item[(D2)] $(1\otimes y)(1\otimes y')=1\otimes yy',$
      \item[(D3)] $(x\otimes 1)(1\otimes y)=x\otimes y,$
      \item[(D4)] $\sum\vf(x_{(2)}, y_{(1)})x_{(1)}\otimes  y_{(2)}=\sum\vf(x_{(1)}, y_{(2)})(1\otimes  y_{(1)})(x_{(2)}\otimes 1),$
      \end{itemize}
      where we use Sweedler’s notation $\Delta^+(x)=\sum x_{(1)}\otimes x_{(2)}$ and $\Delta^-(y)=\sum y_{(1)}\otimes y_{(2)}$. The algebra $\Hall^{\ex}_{\tw}(\cA)^+\widehat{\otimes}\Hall^{\ex}_{\tw}(\cA)^-$ is called the Drinfeld double Hall algebra of $\cA$ , denoted by $\DDHall(\cA)$. Furthermore, $\DDHall(\cA)$ is a bialgebra with counit $\epsilon\otimes\epsilon$ and coproduct $\Delta(a\otimes b)=\Delta^+(a)\otimes\Delta^-(b)$. 
      \end{theorem}

\subsection{$\bZ/2$-graded complexes}\label{subsec: SDH alg}

     Let $\cC_{\bZ/2}(\cA)$ be the category of $\bZ/2$-graded complexes over $\cA$. Namely, an object of such category is
     \[
     \begin{tikzcd}
     M^\bullet:M^0 \arrow[r,->, >=Stealth, shift left, "d_M^0"] & M^1 \arrow[l,->, >=Stealth,shift left, left,"d_M^1"],
     \end{tikzcd}
     \]
     with $d^0_Md^1_M=0,d^1_Md^0_M=0$, where $M^0,M^1\in\cA$, $d^0_M\in\Hom_\cA(M^0,M^1)$, $d^0_M\in\Hom_\cA(M^1,M^0)$, $d^0_M,d^1_M$ are called differentials of $M^\bullet$. The $i$-th homology group $\HG^i(M^\bullet)$ of $M^\bullet$ is defined to be
     $$\HG^i(M^\bullet)=\ker d^i_M/\im d^{i-1}_M,$$
     for any $i\in\bZ/2\bZ$. A morphism $f^\bullet=(f^0,f^1):M^\bullet\rightarrow N^\bullet$ is a pair of morphisms in $\cA$ which fit into the following  diagram
     \[
     \begin{tikzcd}
         M^0 \arrow[r,->, >=Stealth, shift left, "d_M^0"]\arrow[d,->, >=Stealth,"f^0"] & M^1 \arrow[l,->, >=Stealth, shift left,"d_M^1"]\arrow[d,->, >=Stealth,"f^1"]\\
         N^0 \arrow[r,->, >=Stealth, shift left, "d_N^0"] & N^1 \arrow[l,->, >=Stealth, shift left,"d_N^1"],\\
     \end{tikzcd}
     \]
      such that $f^1d_M^0=d_N^0f^0$, $f^0d_M^1=d_N^1f^1$. It is self-evident that $\cC_{\bZ/2}(\cA)$ is an abelian category. In $\cC_{\bZ/2}(\cA)$, for a $\bZ/2$-graded complex $M^\bullet$, the action of the shift functor $\star$ on $M^\bullet$ is defined as $(M^\bullet)^\star=(M^1\rightleftarrows M^0)$ which shifts the grading and changes the signs of the differentials of $M^\bullet$. For $X\in\cA$, we denote
     {\[
     \begin{tikzcd}
         K_X:=(~X \arrow[r,->, >=Stealth,shift left, "1"] & X~), \arrow[l,->, >=Stealth,shift left, left,"0"]&K^\star_X:=(~X \arrow[r,->, >=Stealth,shift left, "0"] & X~), \arrow[l,->, >=Stealth,shift left, left,"1"]
         \\
         C_X:=(~0 \arrow[r,->, >=Stealth,shift left, ""] & X~), \arrow[l,->, >=Stealth,shift left, left,""]&C^\star_X:=(~X \arrow[r,->, >=Stealth,shift left, ""] & 0~). \arrow[l,->, >=Stealth,shift left, left,""],
     \end{tikzcd}
     \]}


     
    The component-wise Euler form $\la-,-\ra_{\cw}:\Iso(\cC_{\bZ/2}(\cA))\times\Iso(\cC_{\bZ/2}(\cA))\rightarrow\bC$ is defined to be 
    \[\la [M^\bullet],[N^\bullet] \ra_{\cw}=\la M^0,N^0 \ra+\la M^1,N^1 \ra,\]
    for $[M^\bullet], [N^\bullet]\in\Iso(\cC_{\bZ/2}(\cA))$. The component-wise Euler form clearly induces  a bilinear form on $K_0(\cC_{\bZ/2}(\cA))$,  which we still denote by $\la-,-\ra_{\cw}$. We denote the representative of $M^\bullet$ in $K_0(\cC_{\bZ/2}(\cA))$ by $\wh{M^\bullet}$, thus $\la [M^\bullet],[N^\bullet] \ra_{\cw}$ can be rewritten as $\la \wh{M^\bullet},\wh{N^\bullet} \ra_{\cw}$.  

     \subsection{semi-derived Ringel-Hall algebras} Let $\Hall(\cC_{\bZ/2}(\cA))$ be the Hall algebra of $\cC_{\bZ/2}(\cA)$ with product given by 
$$[M^\bullet]\ast[N^\bullet]=\sum\limits_{[R^\bullet]\in\Iso({\cC_{\bZ/2}(\cA)})}\upsilon^{\la \widehat{M^\bullet},\widehat{N^\bullet} \ra_{\cw}}h^{R^\bullet}_{M^\bullet N^\bullet}[R^\bullet].$$

     The algebra $\Hall(\cC_{\bZ/2}(\cA))$ is well-known to be a $K_0(\cC_2(\cA))$-graded algebra. Let $I_{\bZ/2}$ be the two-sided ideal generated by the set
     
     $$\{ [L]-[K\oplus M]|0\rightarrow K\rightarrow L\rightarrow M\rightarrow 0\text{,~with K acyclic~} \}.$$
     The quotient algebra $\Hall(\cC_{\bZ/2}(\cA))/I_{\bZ/2}$ is also a $K_0(\cC_2(\cA))$-graded algebra. Let $S_{\bZ/2}$ be the subset of $\Hall(\cC_{\bZ/2}(\cA))/I_{\bZ/2}$ consisting of $t[K]$, where $t\in\bC^\times$ and $[K]$ is an acyclic $\bZ/2$-graded complex. In \cite{LP20}, it has been proven that $S_{\bZ/2}$ is a right Ore set. As a result, we can form the right localization of $\Hall(\cC_{\bZ/2}(\cA))/I_{\bZ/2}$ with respect to $S_{\bZ/2}$.

     \begin{definition}
     
         For any hereditary, abelian, essential small, finitary $\bF_q$-category $\cA$, the algebra $(\Hall(\cC_{\bZ/2}(\cA))/I_{\bZ/2})[S^{-1}_{\bZ/2}]$, which is the right localization of the $\Hall(\cC_{\bZ/2}(\cA))/I_{\bZ/2}$ with respect to the multiplicative set $S_{\bZ/2}$, 
         %
         %
         is called the ($\bZ/2$-graded) semi-derived Ringel-Hall algebra of $\cA$, denoted by $\SHall_{\bZ/2}(\cA)$ . 
     \end{definition}

       For any $\alpha \in K_0(A)$, we define $K_\alpha = [K_A] \ast [K_B]^{-1}$ and $K^\ast_\alpha = [K^\ast_A] \ast [K^\ast_B]^{-1}$, where $\alpha = \widehat{A} - \widehat{B}$ for some $A,B \in A$. This is well-defined, i.e., it does not depend on the choice of $A,B$ such that $\alpha = \widehat{A} - \widehat{B}$.

     \begin{lemma}\cite{LP20}\label{lemma: commutative of K and M}
     In the algebra $\SHall_{\bZ/2}(\cA)$, we have
     \[ K_\alpha\ast[C_A\oplus C_B^\star]=\upsilon^{(\wh{A}-\wh{B},\alpha)}[C_A\oplus C_B^\star]\ast K_\alpha,\]
     \[K_\beta^\star\ast[C_A\oplus C_B^\star]=\upsilon^{(\wh{B}-\wh{A},\beta)}[C_A\oplus C_B^\star]\ast K_\beta^\star,\]
     \[[K_\alpha,K_\beta]=[K_\alpha,K_\beta^\star]=[K_\alpha^\star,K_\beta^\star]=0,\]
     for arbitrary $\alpha,\beta\in K_0(\cA)$ and $A,B\in\cA$.
     
     \end{lemma}

     \begin{theorem}\label{thm: basis of MHA}
         The semi-derived Ringel-Hall algebra $\SHall_{\bZ/2}(\cA)$ of an abelian hereditary category $\cA$ has a basis given by 
         \[[K_\alpha]\ast[K^\star_\beta]\ast[C^\star_{A}\oplus C_B].\]
         where $\alpha,\beta\in K_0(\cA)$ and $A,B\in\cA$. Moreover, for any $[M^\bullet]=(M^i,d^i,i=0,1)\in\cC_{\bZ/2}(\cA)$, we have 

             $$[M^\bullet]=\upsilon^{\la\widehat{\im d^0},\widehat{\HG^0(M^\bullet)}-\widehat{\HG^0(M^\bullet)}\ra+\la\widehat{\im d^1},\widehat{\HG^1(M^\bullet)}-\widehat{\HG^0(M^\bullet)}\ra}K_{\widehat{\im d^0}}\ast K^\star_{\widehat{\im d^1}}\ast[C^\star_{\HG^0(M^\bullet)}\oplus C_{\HG^1(M^\bullet)}],$$
     \end{theorem}
     The following theorem gives the structure constants of the product of this basis.

     \begin{theorem}\cite{Wu18}\label{thm: product formula}
         Let $X_1,Y_1,X_2,Y_2$ be all in the category $\cA$. 
         Then we have 
          \begin{eqnarray*}
          &[C^\star_{X_1}\oplus C_{Y_1}]\ast[C^\star_{X_2}\oplus C_{Y_2}]&=\sum\limits_{\substack{\delta^0,\delta^1\in K_0(\cA)\\K,L\in \Iso(\cA)}}\sum\limits_{\substack{S^0,S^1,M^0,M^1\\N^0,N^1\in\Iso(\cA)\hat{S}^0=\delta^0,\hat{S}^1=\delta^1}}\upsilon^{2N} \frac{a_{X_1}a_{X_2}a_{Y_1}a_{Y_2}}{a_{S^0}a_{S^1}a_{M^0}a_{M^1}a_{N^0}a_{N^1}} \\
          &&\cdot h^{X_2}_{N^1S^1}h^K_{M^0N^1}h^{X_1}_{S^0M^0}h^{Y_2}_{N^0S^0}h^{L}_{M^1N^0}h^{Y_1}_{S^1M^1} K_{\delta^0}\ast K^\star_{\delta^1}\ast[C^\star_{K}\oplus C_{L}],\\
          \end{eqnarray*}
         where $N=\la \delta^0,\widehat{K} \ra+\la \delta^1,\widehat{L} \ra+\la M^0,N^1 \ra+\la M^1,N^0 \ra-\la \delta^0,\widehat{L} \ra-\la \delta^1,\widehat{K} \ra-\la Y_1,\delta^0 \ra-\la X_1,\delta^1\ra-\la \delta^0,\widehat{N^1} \ra-\la \delta^1,\widehat{N^0} \ra.$
     \end{theorem}

     
     By Theorem \ref{thm: basis of MHA} \cite{LP20}, $\SHall_{\bZ/2}(\cA)$ also has a basis given by 
     \[\{[C_A]\ast [C_B^\star]\ast K_\alpha\ast K_\beta^\star\|[A],[B]\in\Iso{\cA} \text{ and } \alpha,\beta\in K_0(\cA)\}.\]
     And by \cite{LP20,Bri13,JX97,Gor13}, there is an embedding of algebras:
     \[I^+:\Hall^{\ex}_{\cw}(\cA)^+\hookrightarrow\SHall_{\bZ/2}(\cA)\]
     \[[A]\mapsto[C_A],\qquad k_\alpha\mapsto K_\alpha,\]
     where $A\in\cA, \alpha\in K_0(\cA)$. By composing $I^+$ and the shift functor $\star$, we also have an embedding
     \[I^-:\Hall^{\ex}_{\cw}(\cA)^-\hookrightarrow\SHall_{\bZ/2}(\cA)\]     \[[A]\mapsto[C_A^\star],\qquad k_\alpha\mapsto K^\star_\alpha,\]     
     which, together with $I^+$, provides an isomorphism of vector spaces $\SHall_{\bZ/2}(\cA)\simeq\DDHall(\cA)$. Furthermore, we have the following theorem.
     
     \begin{theorem}\label{thm: BHA and MHA}\cite{LP20}
     Assume that $\cA$ is a hereditary, abelian, essentially small, finitary $\bF_q$-linear category. There is an isomorphism of algebras:
     \[I:\DDHall(\cA)\rightarrow\SHall_{\bZ/2}(\cA),\]
     defined on the generators by $\sum [A]k_\alpha\otimes [B]k_\beta\mapsto \sum I^+([A]k_\alpha)\ast I^-([B]k_\beta)=\sum [C_A]\ast K_\alpha\ast [C_B^\star]\ast K_\beta^\star$, where $A,B\in\cA$ and $\alpha,\beta\in K_0(\cA)$.

     \end{theorem}

\section{Bialgebra structure on semi-derived Ringel-Hall algebras}\label{sec: bistructure}

     In this section, we define the coproduct of semi-derived Ringel-Hall algebra by using the basis indicated in Theorem \ref{thm: basis of MHA}. We also claim that the defined coproduct is a morphism of algebras, implying that $\SHall_{\bZ/2}(\cA)$ has a bialgebra structure, and at the end of this section, we show that the bialgebra structure of $\SHall_{\bZ/2}(\cA)$ is identified with the bialgebra structure of Drinfeld double Hall algebra. The next section is devoted to proving Lemma \ref{lemma:coassociative law}, Lemma \ref{lemma:coproduct is a morphism of algebra}.

     According to Theorem \ref{thm: basis of MHA}, it suffices to define the coproduct formula for $[C^\star_{A}\oplus C_B]$, $K_\alpha$ and $K_\alpha^\star$ separately if we intend to define an explicit coproduct formula on $\SHall_{\bZ/2}(\cA)$ for an arbitrary hereditary category $\cA$.

     \begin{definition}\label{def:coproduct of MHA}
         Let $\cA$ be a hereditary, abelian, essential small, finitary $\bF_q$-category. We define a linear map $\Delta:\SHall_{\bZ/2}(\cA)\rightarrow\SHall_{\bZ/2}(\cA)\wh{\otimes}\SHall_{\bZ/2}(\cA)$ as follows:
         
         \begin{itemize}
         \item[(1)] For $\alpha\in K_0(\cA)$, $\Delta(K_\alpha)=K_\alpha\otimes K_\alpha$ and $\Delta(K^\star_\alpha)=K^\star_\alpha\otimes K^\star_\alpha$;
         \item[(2)]  For $X,Y\in\cA$, 
             \begin{eqnarray*}
             &&\Delta([C^\star_X\oplus C_Y])\\
             &&=\sum\limits_{\substack{X_1,T,X_2\\Y_1,Y_2}}a_T^{-1}\upsilon^{N}h^X_{X_2TX_1}h^Y_{Y_1TY_2}\frac{a_X}{a_{X_2}a_{X_1}}\cdot\frac{a_Y}{a_{Y_2}a_{Y_1}}[C^\star_{X_1}\oplus C_{Y_1}]K_{Y_2+T}\otimes[C^\star_{X_2}\oplus C_{Y_2}]K^\star_{X_1+T},
             \end{eqnarray*}
         where $N=\la X_2,X_1+T \ra+\la Y_1,Y_2+T \ra+\la T, Y_2 \ra-( X_1+T,Y_2 )-\la X_1,T \ra$. 
We also define a linear map $\epsilon: \SHall_{\bZ/2}(\cA)\rightarrow \bC$  by $\epsilon([M])=\delta_{[M],[0]}$ for $ M\in\cA$.
         \end{itemize}

     \end{definition}

     The following lemmas prove that $(\SHall_{\bZ/2}(\cA),\Delta,\epsilon)$ is a bialgebra.
     
     \begin{lemma}[Coassociative Law]\label{lemma:coassociative law}
         For any hereditary, abelian, essentially small, finitary $\bF_q$-category $\cA$, there is a commutative diagram:
         {\[
         \begin{tikzpicture}[>=stealth,scale=1.2]
         \node (A) at (0,1.5) {$\SHall_{\bZ/2}(\cA)$};
         \node (B) at (6,1.5) {$\SHall_{\bZ/2}(\cA)\otimes\SHall_{\bZ/2}(\cA)$};
         \node (C) at (0,0) {$\SHall_{\bZ/2}(\cA)\otimes\SHall_{\bZ/2}(\cA)$};
         \node (D) at (6,0) {$\SHall_{\bZ/2}(\cA)\otimes\SHall_{\bZ/2}(\cA)\otimes\SHall_{\bZ/2}(\cA)$};
         \draw [->] (B) -- node[right] {$\Delta\otimes$1} (D); 
         \draw [->] (C) -- node[below] {1$\otimes\Delta$} (D);
         \draw [->] (A) -- node[above] {$\Delta$} (B);
         \draw [->] (A) -- node[left] {$\Delta$} (C); 

         \end{tikzpicture}
         \]}
     \end{lemma}

     \begin{lemma}[Compatibility with product]\label{lemma:coproduct is a morphism of algebra}
         For any hereditary, abelian, essential small, finitary $\bF_q$-category $\cA$, the coproduct $\Delta$ defined in the Definition \ref{def:coproduct of MHA} is a morphism of algebras.
     \end{lemma}

     Therefore, we can get our main result from the preceding statements:

     \begin{theorem}
     For any hereditary, abelian, essentially small, finitary $\bF_q$-category $\cA$, $(\SHall_{\bZ/2}(\cA),\Delta,\epsilon)$ is a bialgebra.\label{thm: bialgebra}
     \end{theorem}
     
     In conclusion, we can state the following corollary:
     
     \begin{corollary}
     Assume that $\cA$ is a hereditary, abelian, essentially small, finitary $\bF_q$-category . There is an isomorphism of bialgebras:
     \[I:\DDHall(\cA)\rightarrow\SHall_{\bZ/2}(A)\]
     defined on the generators by $\sum [A]k_\alpha\otimes [B]k_\beta\mapsto \sum I^+(\sum [A]k_\alpha)\ast I^-([B]k_\beta)=\sum [C_A]\ast K_\alpha\ast [C_B^\star]\ast K_\beta^\star$ whenever $A,B\in\cA$ and $\alpha,\beta\in K_0(\cA)$.
       
     \end{corollary}

    \begin{proof}
    According to the preceding definition, it suffices to prove
    \begin{align}
    \label{eq:iso}
    I\otimes I(\Delta([A]\otimes[B]))=\Delta([C_A]\ast[C_B^\star]),\forall A,B\in\cA.\end{align}
    Since $\DDHall(\cA)$ is a bialgebra and $I$ is a morphism of algebras, the left hand side of equation \eqref{eq:iso} is given as follows
    \begin{eqnarray*}
    LHS\eqref{eq:iso}&=&I\otimes I(\Delta([A]\otimes 1)\ast(\Delta(1\otimes[B]))\\
    &=&I\otimes I(\Delta([A]\otimes 1))\ast I\otimes I(\Delta(1\otimes[B]).
    \end{eqnarray*}
    According to Lemma \ref{lemma:coproduct is a morphism of algebra},  $\Delta([C_A]\ast[C_B^\star])=\Delta([C_A])\ast\Delta([C_B^\star])$. Therefore, it suffices to show that
    \[I\otimes I(\Delta([A]\otimes 1))=\Delta([C_A])\]
    and
    \[I\otimes I(\Delta(1\otimes[B]))=\Delta([C_B^\star]).\]
    Given that $\Delta([A])=\sum \upsilon^{\la \widehat{A_1},\widehat{A_2} \ra} h_{A_1A_2}^A\displaystyle{\frac{a_A}{a_{A_1}a_{A_2}}} [A_1]k_{A_2}\otimes [A_2]$, we have
    \[\Delta([A]\otimes 1)=\sum\limits_{A_1,A_2} \upsilon^{\la \widehat{A_1},\widehat{A_2} \ra} h_{A_1A_2}^A\frac{a_A}{a_{A_1}a_{A_2}} ([A_1]k_{A_2}\otimes 1)\otimes ([A_2]\otimes 1).\]
    It follows by
    \[I\otimes I(\Delta([A]\otimes 1))=\sum\limits_{A_1,A_2} \upsilon^{\la \widehat{A_1},\widehat{A_2} \ra} h_{A_1A_2}^A\frac{a_A}{a_{A_1}a_{A_2}} [C_{A_1}]K_{A_2}\otimes [C_{A_2}]=\Delta([C_A]).\]
    The proof of $I\otimes I(\Delta(1\otimes[B]))=\Delta([C_B^\star])$ is similar. 
    \end{proof}
                                                        
\section{Proofs of  Lemmas \ref{lemma:coassociative law}, \ref{lemma:coproduct is a morphism of algebra}}\label{sec: proof}

\subsection{Proof of Lemma \ref{lemma:coassociative law}}

     By Definition \ref{thm: basis of MHA}, it suffices to prove that for any $[C^\star_{X}\oplus C_Y]\in\SHall_{\bZ/2}(\cA)$, we have
     \begin{align}
     \label{eq:coass}
     \Delta\otimes\id(\Delta([C^\star_{X}\oplus C_Y]))=\id\otimes\Delta(\Delta([C^\star_{X}\oplus C_Y])). 
     \end{align}
     According to  Definition \ref{def:coproduct of MHA}, the right hand side of equation \eqref{eq:coass} is as follows:
    \begin{eqnarray*}
     \setlength\abovedisplayskip{.5pt}
     \setlength\belowdisplayskip{.5pt}
     &&RHS\eqref{eq:coass}=\id\otimes\Delta(\Delta([C^\star_X\oplus C_Y]))\\
     &=&\id\otimes\Delta\left(\sum\limits_{\substack{X_1,T,X_2\\Y_1,Y_2}}\upsilon^{n_1}h^X_{X_2TX_1}h^Y_{Y_1TY_2}\frac{a_X}{a_{X_2}a_{X_1}}\cdot\frac{a_Y}{a_{Y_2}a_Ta_{Y_1}}[C^\star_{X_1}\oplus C_{Y_1}]K_{Y_2+T}\otimes[C^\star_{X_2}\oplus C_{Y_2}]K^\star_{X_1+T}\right)\\
     &=&\sum\limits_{\substack{X^1_2,S,X^2_2\\Y^1_2,Y^2_2}}\sum\limits_{\substack{X_1,T,X_2\\Y_1,Y_2}}\upsilon^{n_2}h^X_{X_2TX_1}h^Y_{Y_1TY_2}h^{X_2}_{X^2_2SX^1_2}h^{Y_2}_{Y^1_2SY^2_2}\frac{a_X}{a_{X_2}a_{X_1}}\cdot\frac{a_Y}{a_{Y_2}a_Ta_{Y_1}}\frac{a_{X_2}}{a_{X^2_2}a_{X^1_2}}\cdot\frac{a_{Y_2}}{a_{Y^2_2}a_Sa_{Y^1_2}}\\
     &&\cdot [C^\star_{X_1}\oplus C_{Y_1}]K_{Y_2+T}\otimes[C^\star_{X_2^1}\oplus C_{Y_2^1}]K^\star_{X_1+T}K_{Y_2^2+S}\otimes[C^\star_{X_2^2}\oplus C_{Y_2^2}]K^\star_{X_1+T+X_2^1+S} \\
      \end{eqnarray*}
\begin{align}
      \label{eq: row 5 of RHS of coass}
      &=\sum\limits_{\substack{X^1_2,S,X^2_2\\Y^1_2,Y^2_2}}\sum\limits_{\substack{X_1,T,Y_1}}\upsilon^{n_3}h^X_{X^2_2SX^1_2TX_1}h^Y_{Y_1TY^1_2SY^2_2}\frac{a_X}{a_{X_2^2}a_{X_2^1}a_{X_1}}\cdot\frac{a_Y}{a_{Y^2_2}a_Sa_{Y^1_2}a_Ta_{Y_1}}\\
     &\cdot[C^\star_{X_1}\oplus C_{Y_1}]\ast K_{Y_2+T}\otimes[C^\star_{X_2^1}\oplus C_{Y_2^1}]K^\star_{X_1+T}K_{Y_2^2+S}\otimes[C^\star_{X_2^2}\oplus C_{Y_2^2}]K^\star_{X_1+T+X_2^1+S},\notag \\\notag
     \end{align}
     where 
     \begin{eqnarray*}
     &n_1=&\la \widehat{X_2},\widehat{X_1}+\widehat{T} \ra+\la \widehat{Y_1},\widehat{Y_2}+\widehat{T} \ra+\la \widehat{T},\widehat{Y_2 }\ra-\la \widehat{X_1}+\widehat{T},\widehat{Y_2} \ra-\la \widehat{Y_2},\widehat{X_1}+\widehat{T} \ra-\la \widehat{X_1},\widehat{T} \ra,\\[1mm]
     &n_2=&\la \widehat{X_2},\widehat{X_1}+\widehat{T} \ra+\la \widehat{Y_1},\widehat{Y_2}+\widehat{T} \ra+\la \widehat{T},\widehat{Y_2 }\ra-\la\widehat{X_1}+\widehat{T},\widehat{Y_2} \ra-\la \widehat{Y_2},\widehat{X_1}+\widehat{T} \ra-\la \widehat{X_1},\widehat{T} \ra\\[1mm]
     &&+\la \widehat{X^2_2},\widehat{X^1_2}+\widehat{S} \ra+\la \widehat{Y^1_2},\widehat{Y^2_2}+\widehat{S} \ra+\la \widehat{S},\widehat{Y_2^2} \ra-\la \widehat{X^1_2}+\widehat{S},\widehat{Y^2_2} \ra-\la \widehat{Y^2_2},\widehat{X^1_2}+\widehat{S} \ra-\la \widehat{X^1_2},\widehat{S} \ra,\\[1mm]
     &n_3=&\la \widehat{X_2^2}\!+\!\widehat{S}\!+\!\widehat{X_2^1},\widehat{X_1}\!+\!\widehat{T} \ra\!+\!\la \widehat{Y_1},\widehat{Y_2^1}\!+\!\widehat{S}\!+\!\widehat{Y_2^2}\!+\!\widehat{T} \ra\!+\!\la \widehat{T },\widehat{Y_2^1}\!+\!\widehat{S}\!+\!\widehat{Y_2^2}\ra\!-\!\la\widehat{X_1}\!+\!\widehat{T},\widehat{Y_2^1}\!+\!\widehat{S}\!+\!\widehat{Y_2^2} \ra\\[1mm]
     &&\!-\!\la \widehat{Y^1_2}\!+\!\widehat{S}\!+\!\widehat{Y_2^2},\widehat{X_1}\!+\!\widehat{T} \ra\!-\!\la \widehat{X_1},\widehat{T} \ra\!+\!\la \widehat{X^2_2},\widehat{X^1_2}\!+\!\widehat{S} \ra\!+\!\la \widehat{Y^1_2},\widehat{Y^2_2}\!+\!\widehat{S} \ra\!+\!\la \widehat{S},\widehat{Y^2_2} \ra\!-\!\la \widehat{X^1_2}\!+\!\widehat{S},\widehat{Y^2_2} \ra\\[1mm]
     &&\!-\!\la \widehat{Y^2_2},\widehat{X^1_2}\!+\!\widehat{S} \ra\!-\!\la \widehat{X^1_2},\widehat{S} \ra.
     \end{eqnarray*}
     And the equation (\ref{eq: row 5 of RHS of coass}) follows from the relation $\widehat{X_2}=\widehat{X_2^1}+\widehat{S}+\widehat{X_2^2}$ and $\widehat{Y_2}=\widehat{Y_2^1}+\widehat{S}+
     \widehat{Y_2^2}$ in $K_0(\cA)$.
     We denote $\widehat{P_X}=\widehat{X_2^1}+\widehat{T}+\widehat{X_1}$, $\widehat{P_Y}=\widehat{Y_1}+\widehat{T}+\widehat{Y_2^1}$,  $m_1=\la \widehat{X_2^2},\widehat{P_X}+\widehat{S}\ra+\la \widehat{P_Y},\widehat{S}+\widehat{Y_2^2} \ra+\la \widehat{S},\widehat{Y^2_2} \ra-\la \widehat{P_X}+\widehat{S},\widehat{Y^2_2} \ra-\la \widehat{Y^2_2},\widehat{P_X}+\widehat{S} \ra-\la\widehat{P_X},\widehat{S} \ra$ and $m_2=\la \widehat{X^1_2},\widehat{X_1}+\widehat{T} \ra+\la \widehat{Y_1},\widehat{Y^1_2}+\widehat{T} \ra+\la\widehat{T}, \widehat{Y_2^1}\ra-\la\widehat{X_1}+\widehat{T},\widehat{Y_2^1}\ra-\la\widehat{Y_2^1},\widehat{X_1}+\widehat{T}\ra-\la \widehat{X_1},\widehat{T} \ra$. It is not hard to check that $n_3=m_1+m_2$, by replacing $n_3$ with $m_1+m_2$, the equation (\ref{eq: row 5 of RHS of coass}) is equal to
     \begin{eqnarray*}
     \setlength\abovedisplayskip{.5pt}
     \setlength\belowdisplayskip{.5pt}
     &=&\sum\limits_{\substack{P_X,S,X^2_2\\P_Y,Y^2_2}}\sum\limits_{\substack{X_2^1,T,X_1\\Y_1,Y_2^1}}\upsilon^{m_1+m_2}h^X_{X^2_2SP_X}h^Y_{P_YSY^2_2}h^{P_X}_{X_2^1TX_1}h^{P_Y}_{Y_1TY_2^1}\frac{a_X}{a_{X_2}a_{X_1}}\cdot\frac{a_Y}{a_{Y_2}a_Ta_{Y_1}}\frac{a_{X_2}}{a_{X^2_2}a_Sa_{X^1_2}}\\
     &&\cdot\frac{a_{Y_2}}{a_{Y^2_2}a_{Y^1_2}}[C^\star_{X_1}\oplus C_{Y_1}]K_{Y^1_2+T+Y_2^2+S}\otimes[C^\star_{X_2^1}\oplus C_{Y_2^1}]K^\star_{X_1+T}K_{Y_2^2+S}\otimes[C^\star_{X_2^2}\oplus C_{Y_2^2}]K^\star_{P_X+S} \\
     \end{eqnarray*}
     \begin{eqnarray*}
     \setlength\abovedisplayskip{.5pt}
     \setlength\belowdisplayskip{.5pt}
     &=&\sum\limits_{\substack{P_X,S,X^2_2\\P_Y,Y^2_2}}\sum\limits_{\substack{X_2^1,T,X_1\\Y_1,Y_2^1}}\upsilon^{m_1+m_2}h^X_{X^2_2SP_X}h^Y_{P_YSY^2_2}h^{P_X}_{X_2^1TX_1}h^{P_Y}_{Y_1TY_2^1}\frac{a_X}{a_{X^2_2}a_Sa_{P_X}}\cdot\frac{a_Y}{a_{P_Y}a_{Y^2_2}}\frac{a_{P_X}}{a_{X^1_2}a_Ta_{X_1}}\\
     &&\cdot\frac{a_{P_Y}}{a_{Y_1}a_{Y^1_2}}[C^\star_{X_1}\oplus C_{Y_1}]K_{Y^1_2+T+Y_2^2+S}\otimes[C^\star_{X_2^1}\oplus C_{Y_2^1}]K^\star_{X_1+T}K_{Y_2^2+S}\otimes[C^\star_{X_2^2}\oplus C_{Y_2^2}]K^\star_{P_X+S}. \\
     \end{eqnarray*}

 Due to $\widehat{X}=\widehat{P_X}+\widehat{S}+\widehat{X^2_2}$ and $\widehat{Y}=\widehat{Y_2^2}+\widehat{S}+\widehat{P_Y}$, it follows that $\Delta([C^\star_X\oplus C_Y])$ and $\Delta([C^\star_{P_X}\oplus C_{P_Y}])$ can be rewritten as:
     \begin{align*}
     &\Delta([C^\star_X\oplus C_Y])=\sum\limits_{\substack{P_X,S,X_2^2\\P_Y,Y_2^2}}\upsilon^{m_1}h^X_{X_2^2SP_X}h^Y_{P_YSY_2^2}\frac{a_X}{a_{X^2_2}a_{P_X}}\cdot\frac{a_Y}{a_{P_Y}a_Sa_{Y^2_2}}[C^\star_{P_X}\oplus C_{P_Y}]K_{Y_2^2+S}\otimes[C^\star_{X_2^2}\oplus C_{Y_2^2}]K^\star_{P_X+S},\\
     &\Delta([C^\star_{P_X}\oplus C_{P_Y}])=\sum\limits_{\substack{X_2^1,T,X_1\\Y_1,Y_2^1}}\upsilon^{m_2}h^{P_X}_{X_2^1TX_1}h^{P_Y}_{Y_1TY_2^1}\frac{a_{P_X}}{a_{X^1_2}a_{X_1}}\cdot\frac{a_{P_Y}}{a_{Y_1}a_Ta_{Y^1_2}}[C^\star_{X_1}\oplus C_{Y_1}]K_{Y_2^1+T}\otimes[C^\star_{X_2^1}\oplus C_{Y_2^1}]K^\star_{X_1+T}.
     \end{align*}
     Then we get 
     \begin{align*}
     RHS\eqref{eq:coass}=&\sum\limits_{\substack{P_X,S,X^2_2\\P_Y,Y^2_2}}\upsilon^{m_1}h^X_{X^2_2SP_X}h^Y_{P_YSY^2_2}\frac{a_Xa_Y}{a_{X^2_2}a_{P_X}a_{P_Y}a_S a_{Y^2_2}}\Bigg(\sum\limits_{\substack{X_1,T,X_2^1\\Y_1,Y_2^1}}\upsilon^{m_2}\frac{a_{P_X}a_{P_Y}}{a_{X^1_2}a_{X_1}a_{Y_1}a_T a_{Y^1_2}}h^{P_X}_{X_2^1TX_1}h^{P_Y}_{Y_1TY_2^1}\\
     &[C^\star_{X_1}\oplus C_{Y_1}]K_{Y^1_2+T}K_{Y_2^2+S}\otimes[C^\star_{X_2^1}\oplus C_{Y_2^1}]K^\star_{X_1+T}K_{Y_2^2+S}\Bigg)\otimes[C^\star_{X_2^2}\oplus C_{Y_2^2}]K^\star_{P_X+S}\\
     =&\Delta\!\otimes\!\id\left(\sum\limits_{\substack{X_1,T,X_2\\Y_1,Y_2}}\!\upsilon^{m_1}\frac{a_X}{a_{X^2_2}a_{P_X}}\cdot\frac{a_Y}{a_{P_Y}a_Sa_{Y^2_2}}h^X_{X^2_2SP_X}h^Y_{P_YSY^2_2}[C^\star_{P_X}\!\oplus\!C_{P_Y}]K_{Y_2^2\!+\!S}\!\otimes\![C^\star_{X_2^2}\!\oplus\!C_{Y_2^2}]K^\star_{P_X\!+\!S} \right)\\
     =&\Delta\otimes\id\left( \Delta([C^\star_X\oplus C_Y]) \right)\\
     =&LHS\eqref{eq:coass},
     \end{align*}
     where LHS \eqref{eq:coass} represents left-hand side of \eqref{eq:coass}. The proof is complete.

\subsection{Proof of Lemma \ref{lemma:coproduct is a morphism of algebra}}
     By Definition \ref{def:coproduct of MHA}, it suffices to prove that 
     \begin{equation}\label{eqn:coproduct is alg morphism}
     \Delta([C^\star_{X_1}\oplus C_{Y_1}]\ast[C^\star_{X_2}\oplus C_{Y_2}])=\Delta([C^\star_{X_1}\oplus C_{Y_1}])\ast \Delta([C^\star_{X_2}\oplus C_{Y_2}]).
     \end{equation}
     By Theorem \ref{thm: product formula}, we have
     \begin{align}
     &\Delta([C^\star_{X_1}\oplus C_{Y_1}]\ast [C^\star_{X_2}\oplus C_{Y_2}])\notag\\
     &=\Delta\Bigg(\sum\limits_{\substack{\delta^0,\delta^1\in K_0(\cA)\\K,L\in \Iso(\cA)}} \sum\limits_{\substack{S^0,S^1,M^0,M^1\\N^0,N^1\in\Iso(\cA)\\\hat{S}^0=\delta^0,\hat{S}^1=\delta^1}} \upsilon^{n_0}\frac{a_{X_1}a_{X_2}a_{Y_1}a_{Y_2}}{a_{M^1}a_{N^1}a_{M^0}a_{N^0}a_{S^0}a_{S^1}}h^{X_2}_{N^1S^1}h^K_{M^0N^1}h^{X_1}_{S^0M^0}h^{Y_2}_{N^0S^0}h^{L}_{M^1N^0}h^{Y_1}_{S^1M^1}\notag\\
     &\qquad\qquad\upsilon^{n_0}K_{\delta^0}\ast K^\star_{\delta^1}\ast[C^\star_{K}\oplus C_{L}] \Bigg)\notag\\
     &=\sum\limits_{\substack{\delta^0,\delta^1\in K_0(\cA)\\K,L\in \Iso(\cA)}} \sum\limits_{\substack{S^0,S^1,M^0,M^1\\N^0,N^1\in\Iso(\cA)\\\hat{S}^0=\delta^0,\hat{S}^1=\delta^1}} a_T\upsilon^{n_1}\frac{a_{X_1}a_{X_2}a_{Y_1}a_{Y_2}a_Ka_L}{a_{M^1}a_{N^1}a_{M^0}a_{N^0}a_{S^0}a_{S^1}a_{K_1}a_Ta_{K^0}a_{L^0}a_Ta_{L^1}}\sum\limits_{\substack{K^0,T,K^1\\L^0,L^1}}h^{X_2}_{N^1S^1}h^K_{M^0N^1}\label{eqn: l-4.1}\\
     &h^{X_1}_{S^0M^0}h^{Y_2}_{N^0S^0}h^{L}_{M^1N^0}h^{Y_1}_{S^1M^1}h^K_{K^1TK^0}h^L_{L^0TL^1}K_{\delta^0+\widehat{L^1}+\widehat{T}} K^\star_{\delta^1}[C^\star_{K^0}\oplus C_{L^0}]\otimes K_{\delta^0} K^\star_{\delta^1+\widehat{K^0}+\widehat{T}}[C^\star_{K^1}\oplus C_{L^1}],\notag
     \end{align}
     where
     \begin{eqnarray*}
     &n_0=&2(\la \delta^0,\widehat{K} \ra\!+\!\la \delta^1,\widehat{L} \ra\!-\!\la \delta^0,\widehat{L} \ra\!-\!\la \delta^1,\widehat{K} \ra\!+\!\la M^0,N^1 \ra\!+\!\la M^1,N^0 \ra\!-\!\la Y^1,\delta^0 \ra\!-\!\la X^1,\delta^1\ra\!\\
     &&-\!\la \delta^0,\widehat{N^1} \ra\!-\!\la \delta^1,\widehat{N^0}\ra),\\
     &n_1=&2(\la \delta^0,\widehat{K} \ra\!+\!\la \delta^1,\widehat{L} \ra\!-\!\la \delta^0,\widehat{L} \ra\!-\!\la \delta^1,\widehat{K} \ra\!+\!\la M^0,N^1 \ra\!+\!\la M^1,N^0 \ra\!-\!\la Y^1,\delta^0 \ra\!-\!\la X^1,\delta^1\ra\!\\
     &&-\!\la \delta^0,\widehat{N^1} \ra\!-\!\la \delta^1,\widehat{N^0} \ra)+\la \widehat{K^1},\widehat{K^0}\!+\!\widehat{T} \ra\!+\!\la \widehat{L^0},\widehat{L^1}\!+\!\widehat{T} \ra\!+\!(\widehat{L^1}\!+\!\widehat{T},\widehat{K^0})\!+\!\la \widehat{T},\widehat{L^1} \ra\!\\
     &&-\!( \widehat{K^0}\!+\!\widehat{T},\widehat{K^1} )\!-\!(\widehat{L^1}\!+\!\widehat{T},\widehat{L^0})\!-\!\la \widehat{K^0},\widehat{T} \ra;
     \end{eqnarray*}
     Let $A=\sum\limits_{\substack{M^0,M^1\\N^0,N^1}}\sum\limits_{K,L,T}\frac{a_Ka_L}{a_M^1a_N^1a_M^0a_N^0} h^{X_2}_{N^1S^1}h^K_{M^0N^1}h^{X_1}_{S^0M^0}h^{Y_2}_{N^0S^0}h^{L}_{M^1N^0}h^{Y_1}_{S^1M^1}h^K_{K^1TK^0}h^L_{L^0TL^1}$, together with Corollary \ref{cor: green formula-3,2}, we have
     \begin{align*}
     &A= \sum\limits_{\substack{M^0,M^1\\N^0,N^1}}\sum\limits_{K,L,T}\frac{a_Ka_L}{a_M^1a_N^1a_M^0a_N^0}h^{X_2}_{N^1S^1}h^K_{M^0N^1}h^{X_1}_{S^0M^0}h^{Y_2}_{N^0S^0}h^{L}_{M^1N^0}h^{Y_1}_{S^1M^1}h^K_{K^1TK^0}h^L_{L^0TL^1}\\
     &=\sum\limits_{\substack{M^0,M^1\\N^0,N^1}}h^{X_2}_{N^1S^1}h^{X_1}_{S^0M^0}h^{Y_2}_{N^0S^0}h^{Y_1}_{S^1M^1}\sum\limits_{T}\left(\sum\limits_{\substack{M^0_1,M^0_2,T_1\\N^1_1,N^1_2,T_2}}h^{M^0}_{M^0_1T_1M^0_2}h^{N^1}_{N_1^1T_2N^1_2}h^{K^1}_{M^0_1N^1_1}h^T_{T_1T_2}h^{K^0}_{M^0_2N^1_2}\right)\\
     &\cdot\left(\sum\limits_{\substack{M^1_1,M^1_2,E_1\\N^0_1,N^0_2,E_2}}h^{M^1}_{M^1_1E_1M^1_2}h^{N^0}_{N^0_1E_2N^0_2}h^{L^0}_{M^1_1N^0_1}h^T_{E_1E_2}h^{L^1}_{M_2^1N_2^0}\right)\upsilon^{-2\left(\la\widehat{M^0_1},\widehat{T_2}+\widehat{N_2^1}\ra+\la \widehat{T_1},\widehat{N_2^1} \ra+\la \widehat{E_1},\widehat{E_2}+\widehat{N_2^0} \ra+\la \widehat{E_1},\widehat{N_2^0} \ra\right)}\\[2mm]
     &\cdot\frac{a_{K^0}a_Ta_{K^1}a_{L^0}a_Ta_{L^1}}{a_{M_1^0}a_{T_1}a_{M_2^0}a_{N_1^1}a_{T_2}a_{N_2^1}a_{M_1^1}a_{E_1}a_{M_2^1}a_{N_1^0}a_{E_2}a_{N_2^0}}\\[2mm]
     &=\sum\limits_{\substack{M^0,M^1\\N^0,N^1}}\sum\limits_{\substack{M^0_1,M^0_2,T_1\\N^1_1,N^1_2,T_2}}\sum\limits_{\substack{M^1_1,M^1_2,E_1\\N^0_1,N^0_2,E_2}}\sum\limits_{T_1^0,T_1^1,T_2^0,T_2^1}h^{T_1}_{T^0_1T^1_1}h^{T_2}_{T^0_2T^1_2}h^{E_1}_{T_1^0T_2^0}h^{E_2}_{T_1^1T_2^1}h^{M^0}_{M^0_1T_1M^0_2}h^{N^1}_{N_1^1T_2N^1_2}h^{K^1}_{M^0_1N^1_1}h^{K^0}_{M^0_2N^1_2}\\
     &\cdot h^{M^1}_{M_1^1E_1M_2^1}h^{N^0}_{N^0_1E_2N^0_2}h^{L^0}_{M^1_1N^0_1}h^{L^1}_{M_2^1N_2^0}h^{X_2}_{N^1S^1}h^{X_1}_{S^0M^0}h^{Y_2}_{N^0S^0}h^{Y_1}_{M^1S^1}\!\cdot\!\upsilon^{-2\left(\la\widehat{M^0_1},\widehat{T_2}\!+\!\widehat{N_2^1}\ra\!+\!\la \widehat{T_1},\widehat{N_2^1} \ra\!+\!\la \widehat{E_1},\widehat{E_2}\!+\!\widehat{N_2^0} \ra\!+\!\la \widehat{E_1},\widehat{N_2^0} \ra\right)}\\[5.5mm]
     &\cdot\frac{a_{K^0}a_Ta_{K^1}a_{L^0}a_{L^1}}{a_{M_1^0}a_{M_2^0}a_{N_1^1}a_{N_2^1}a_{M_1^1}a_{M_2^1}a_{N_1^0}a_{N_2^0}a_{T_1^0}a_{T_1^1}a_{T_2^0}a_{T_2^1}}\\
     &\\
     &=\sum\limits_{\substack{M^0_1,M^0_2\\N^1_1,N^1_2}}\sum\limits_{\substack{M^1_1,M^1_2\\N^0_1,N^0_2}}\sum\limits_{T_1^0,T_1^1,T_2^0,T_2^1}h^{X_1}_{S^0M_1^0T_1^0T_1^1M_2^0}h^{Y_1}_{S^1M_1^1T_1^0T_2^0M_2^1}h^{X_2}_{N^1_1T_2^0T_2^1N^1_2S^1}h^{Y_2}_{N^0_1T_1^1T_2^1N^0_2S^0}h_{M_1^0N_1^1}^{K^1}h_{M_2^0N_2^1}^{K^0}h_{M_1^1N_1^0}^{L^0}h_{M_2^1N_2^0}^{L^1}\\
     &\cdot\upsilon^{-2\left(\la\widehat{M^0_1},\widehat{T_2}\!+\!\widehat{N_2^1}\ra\!+\!\la \widehat{T_1},\widehat{N_2^1} \ra\!+\!\la \widehat{T}_1^0 + \widehat{T}_2^0,\widehat{T}_1^1 + \widehat{T}_2^1\!+\!\widehat{N_2^0} \ra\!+\!\la \widehat{T}_1^0 + \widehat{T}_2^0,\widehat{N_2^0} \ra\right)}\!\cdot\!\frac{a_{K^0}a_Ta_{K^1}a_{L^0}a_{L^1}}{a_{M_1^0}a_{M_2^0}a_{N_1^1}a_{N_2^1}a_{M_1^1}a_{M_2^1}a_{N_1^0}a_{N_2^0}a_{T_1^0}a_{T_1^1}a_{T_2^0}a_{T_2^1}}.
     \end{align*}
     Hence, we obtain equation (\ref{eqn: l-4.1}) is equal to 
     \begin{align*}
&\sum\limits_{\substack{S^0,S^1}}\sum\limits_{\substack{M^0_1,M^0_2\\N^1_1,N^1_2}}\sum\limits_{\substack{M^1_1,M^1_2\\N^0_1,N^0_2}}\sum\limits_{T_1^0,T_1^1,T_2^0,T_2^1}h^{X_1}_{S^0M_1^0T_1^0T_1^1M_2^0}h^{Y_1}_{S^1M_1^1T_1^0T_2^0M_2^1}h^{X_2}_{N^1_1T_2^0T_2^1N^1_2S^1}h^{Y_2}_{N^0_1T_1^1T_2^1N^0_2S^0}\\[2.5mm]
     &\cdot\frac{a_{K^0}a_{K^1}a_{L^0}a_{L^1}a_{X_1}a_{X_2}a_{Y_1}a_{Y_2}}{a_{M_1^0}a_{M_2^0}a_{N_1^1}a_{N_2^1}a_{M_1^1}a_{M_2^1}a_{N_1^0}a_{N_2^0}a_{T_1^0}a_{T_1^1}a_{T_2^0}a_{T_2^1}}\!\cdot\!\upsilon^{n_2}K_{\widehat{S^0}\!+\!\widehat{L^1}\!+\!\widehat{T}}\ast K^\star_{\delta^1}\ast[C^\star_{K^0}\oplus C_{L^0}]\\[2.5mm]
     &\otimes K_{\widehat{S^0}}\!\ast\!K^\star_{\widehat{S^1}\!+\!\widehat{K^0}\!+\!\widehat{T}}[C^\star_{K^1}\oplus C_{L^1}]\\
     \end{align*}
     where $n_2=\!2(\la \delta^0,\widehat{K} \ra\!+\!\la \delta^1,\widehat{L} \ra\!-\!\la \delta^0,\widehat{L} \ra\!-\!\la \delta^1,\widehat{K} \ra\!+\!\la M^0,N^1 \ra\!+\!\la M^1,N^0 \ra\!-\!\la Y^1,\delta^0 \ra\!-\!\la X^1,\delta^1\ra\!-\!\la \delta^0,\widehat{N^1} \ra\!-\!\la \delta^1,\widehat{N^0} \ra)+\la \widehat{K^1},\widehat{K^0}+\widehat{T} \ra+\la \widehat{L^0},\widehat{L^1}+\widehat{T} \ra+(\widehat{L^1}+\widehat{T},\widehat{K^0})+\la \widehat{T},\widehat{L^1} \ra-( \widehat{K^0}+\widehat{T},\widehat{K^1} )-(\widehat{L^1}+\widehat{T},\widehat{L^0})-\la \widehat{K^0},\widehat{T} \ra-2\left(\la\widehat{M^0_1},\widehat{T_2}+\widehat{N_2^1}\ra+\la \widehat{T_1},\widehat{N_2^1} \ra+\la \widehat{E_1},\widehat{E_2}+\widehat{N_2^0} \ra+\la \widehat{E_1},\widehat{N_2^0} \ra\right)$. 

     On the other hand, the right hand side of equation (\ref{eqn:coproduct is alg morphism}) is equal to 
     \begin{align*}
     &\Delta([C^\star_{X_1}\oplus C_{Y_1}])\ast \Delta([C^\star_{X_2}\oplus C_{Y_2}])\\
     &=\Bigl(\sum\limits_{\substack{X^0_1,T_1^0,X_1^1\\Y^0_1,Y_1^1}}a_{T_1^0}^{-1}\upsilon^{m_1}h^{X_1}_{X^1_1T^0_1X^0_1}h^{Y_1}_{Y^0_1T^0_1Y_1^1}\frac{a_{X_1}}{a_{X_1^1}a_{X_1^0}}\frac{a_{Y_1}}{a_{Y_1^0}a_{Y_1^1}}[C^\star_{X_1^0}\oplus C_{Y^0_1}]K_{Y^1_1+T_1^0}\otimes[C^\star_{X_1^1}\oplus C_{Y_1^1}]\cdot\\
    &K^\star_{X_1^0+T_1^0}\Bigr)\ast\Bigl(\sum\limits_{\substack{X^0_2,T_2^1,X_2^1\\Y^0_2,Y_2^1}}a_{T_2^1}^{-1}\upsilon^{m_2}h^{X_2}_{X^1_2T_2^1X^0_2}h^{Y_2}_{Y^0_2T_2^1Y_2^1}\frac{a_{X_2}}{a_{X_2^1}a_{X_2^0}}\frac{a_{Y_2}}{a_{Y_2^0}a_{Y_2^1}}[C^\star_{X_2^0}\oplus C_{Y^0_2}] K_{Y^1_2+T_2^1}\otimes[C^\star_{X_2^1}\\
    &\oplus C_{Y_2^1}]K^\star_{X_2^0+T_2^1}\Bigr)\\
     &=\sum\limits_{\substack{X^0_1,T_1^0,X_1^1\\Y^0_1,Y_1^1}}\sum\limits_{\substack{X^0_2,T_2^1,X_2^1\\Y^0_2,Y_2^1}}\upsilon^{m_3}h^{X_1}_{X^1_1T_1^0X^0_1}h^{Y_1}_{Y^0_1T_1^0Y_1^1}h^{X_2}_{X^1_2T_2^1X^0_2}h^{Y_2}_{Y^0_2T_2^1Y_2^1}\\
     &\cdot\left(\sum_{\substack{\delta^0_1,\delta^1_1\in K_0(\cA)\\K^0,L^0}}\sum_{\substack{T_1^1,S^1,M_2^0,M_1^1,\\ N_1^0,N_2^1\in Iso(\cA) \\\widehat{T_1^1}=\delta_1^0, \widehat{S^1}=\delta_1^1}}h_{N_2^1S^1}^{X_2^0}h_{M_2^0N_2^1}^{K^0}h_{T_1^1M_2^0}^{X_1^0}h_{N_1^0T_1^1}^{Y_2^0}h_{M_1^1N_1^0}^{L^0}h_{S^1M_1^1}^{Y_1^0}\frac{a_{X^0_1}a_{X_2^0} a_{Y^0_1}a_{Y_2^0}}{a_{M_2^0}a_{M_1^1}a_{N_1^0}a_{N_2^1}a_{T_1^1}a_{S^1}} \right)\\[2.5mm]
     &\cdot\left(\sum_{\substack{\delta^0_2,\delta^1_2\in K_0(\mathcal{A})\\K^1,L^1}}\sum_{\substack{S^0,T_2^0,M_1^0,M_2^1,\\ N_2^0,N_1^1\in Iso(\cA) \\\hat{S}_2^0=\delta_2^0, \hat{S}_2^1=\delta_2^1}}h_{N_1^1T_2^0}^{X_2^1}h_{M_1^0N_1^1}^{K^1}h_{S^0M_1^0}^{X_1^1}h_{N_2^0S^0}^{Y_2^1}h_{M_2^1N_2^0}^{L^1}h_{T_2^0M_2^1}^{Y_1^1}\frac{a_{X_2^1}a_{X_1^1}a_{Y_2^1}a_{Y_1^1}}{a_{M_1^0}a_{M_2^1}a_{N_2^0}a_{N_1^1}a_{S^0}a_{T_2^0}}\right)\\[2.5mm]
     &\frac{a_{X_1}a_{Y_1}a_{X_2}a_{Y_2}}{a_{X_1^1}a_{X_1^0}a_{Y_1^0}a_{Y_1^1}a_{X_2^1}a_{X_2^0}a_{Y_2^0}a_{Y_2^1}a_{T_1^0}a_{T_2^1}}K_{\widehat{S^0}+\widehat{L^1}+\widehat{T}}\ast K^\star_{\widehat{S^1}}\ast[C^\star_{K^0}\oplus C_{L^0}]\otimes K_{\widehat{S^0}}\ast K^\star_{\widehat{S^1}+\widehat{K^0}+\widehat{T}}[C^\star_{K^1}\\[2.5mm]
     &\oplus C_{L^1}]\\[2.5mm]
    \end{align*}
    \begin{align}
     &=\sum\limits_{\substack{S^0,S^1}}\sum\limits_{\substack{M^0_1,M^0_2\\N^1_1,N^1_2}}\sum\limits_{\substack{M^1_1,M^1_2\\N^0_1,N^0_2}}\sum\limits_{T_1^0,T_1^1,T_2^0,T_2^1}\upsilon^{m_3}h^{X_1}_{S^0M_1^0T_1^0T_1^1M_2^0}h^{Y_1}_{S^1M_1^1T_1^0T_2^0M_2^1}h^{X_2}_{N^1_1T_2^0T_2^1N^1_2S^1}h^{Y_2}_{N^0_1T_1^1T_2^1N^0_2S^0}\label{eqn: r-4.1}\\
     &\cdot\frac{a_{K^0}a_{K^1}a_{L^0}a_{L^1}a_{X_1}a_{X_2}a_{Y_1}a_{Y_2}}{a_{M_1^0}a_{M_2^0}a_{N_1^1}a_{N_2^1}a_{M_1^1}a_{M_2^1}a_{N_1^0}a_{N_2^0}a_{T_1^0}a_{T_1^1}a_{T_2^0}a_{T_2^1}} K_{\widehat{S^0}+\widehat{L^1}+\widehat{T}} K^\star_{S^1}[C^\star_{K^0}\oplus C_{L^0}]\otimes K_{\widehat{S^0}}K^\star_{\widehat{S^1}+\widehat{K^0}+\widehat{T}}\notag\\[2.5mm]
     &\ast[C^\star_{K^1}\oplus C_{L^1}]\notag\\\notag
     \end{align}
     
     where,

     \begin{eqnarray*}
     &m_1=&\la \wh{X_1^1},\wh{X_1^0}+\wh{T_1^0} \ra+\la \wh{Y^0_1},\wh{Y_1^1}+\wh{T_1^0} \ra+\la \wh{T_1^0},\wh{Y_1^1} \ra-( \wh{X_1^0}+\wh{T_1^0},\wh{Y_1^1} )-\la \wh{X_1^0},\wh{T_1^0} \ra,\\[1mm]
     &m_2=&\la \wh{X_2^1},\wh{X_2^0}+\wh{T_2^1} \ra+\la \wh{Y^0_2},\wh{Y_2^1}+\wh{T_2^1} \ra+\la \wh{T_2^1},\wh{Y_2^1} \ra-( \wh{X_2^0}+\wh{T_2^1},\wh{Y_2^1} )-\la \wh{X_2^0},\wh{T_2^1} \ra,\\[1mm]
     &m_3=&m_1+m_2+(\widehat{Y_2^1}+\widehat{T_2^1},X_2^0-Y_2^0)+(\widehat{Y_2^1}+\widehat{T_2^1},X_1^0-Y_1^0)+(\widehat{Y_1^1}+\widehat{T_1^0},\widehat{X_1^0}-\widehat{Y_1^0})\\[1mm]
     &&+(\widehat{X_2^0}+\widehat{T_2^1},\widehat{Y_2^1}-\widehat{X_2^1})+(\widehat{X_2^0}+\widehat{T_2^1},\widehat{Y_1^1}-\widehat{X_1^1})+(\widehat{X_1^0}+\widehat{T_1^0},\widehat{Y_1^1}-\widehat{X_1^1})+2(\la T_1^1,\widehat{K^0} \ra\\[1mm]
     &&+\la S^1,\widehat{L^0} \ra+\la M_2^0,N_2^1 \ra+\la M_1^1,N_1^0 \ra-\la T_1^1,\widehat{L^0} \ra-\la S^1,\widehat{K^0} \ra-\la Y^0_1,T_1^1 \ra-\la X^0_1,S^1\ra\\[1mm]
     &&-\la T_1^1,\widehat{N_2^1} \ra-\la S^1,\widehat{N_2^0} \ra+\la S^0,\widehat{K^1} \ra+\la T_2^0,\widehat{L^1} \ra+\la M_1^0,N_1^1 \ra+\la M_2^1,N_2^0 \ra-\la S^0,\widehat{L^1} \ra\\
     &&-\la T_2^0,\widehat{K^1} \ra-\la Y^1_1,S^0 \ra-\la X^1_1,T_2^0\ra-\la S^0,\widehat{N_1^1} \ra-\la T_2^0,\widehat{N_2^0} \ra);\\
     \end{eqnarray*}

     Comparing equation (\ref{eqn: l-4.1}) and equation (\ref{eqn: r-4.1}), it suffices to prove 
     \[m_3=n_2.\]
     And we can obtain this identity directly by the relations in $K_0(\cA)$:
     \begin{eqnarray*}
     &\wh{X_1}&=\wh{S^0}+\wh{M_1^0}+\wh{T_1^0}+\wh{T_1^1}+\wh{M_2^1},\\
     &\wh{Y_1}&=\wh{S^1}+\wh{M_1^1}+\wh{T_1^0}+\wh{T_2^0}+\wh{M_2^1},\\
     &\wh{X_2}&=\wh{N_1^1}+\wh{T_2^0}+\wh{T_2^1}+\wh{N_2^1}+\wh{S^1},\\
     &\wh{Y_2}&=\wh{N_1^0}+\wh{T_1^1}+\wh{T_2^1}+\wh{N_2^0}+\wh{S^0},\\
     &\wh{T_1}&=\wh{T_1^0}+\wh{T_1^1}, \wh{T_2}=\wh{T_2^0}+\wh{T_2^1},\\
     \end{eqnarray*}

\end{sloppypar}
\end{document}